\documentclass[11pt,reqno]{amsart}
\usepackage{indentfirst, amssymb,amsmath,amsthm, mathrsfs,cite,graphicx,float}
\newtheorem{theo}{Theorem}[section]

\newtheorem{open problem}{Open problem}[section]
\newcommand{\pa}{\partial}
\newcommand{\ol}{\overline}

\newcommand{\be}{\begin{equation}}
\newcommand{\ee}{\end{equation}}
\newcommand{\bs}{\begin{small}}
\newcommand{\es}{\end{small}}
\newcommand{\beas}{\begin{eqnarray*}}
\newcommand{\eeas}{\end{eqnarray*}}
\newcommand{\bea}{\begin{eqnarray}}
\newcommand{\eea}{\end{eqnarray}}
\renewcommand{\epsilon}{\varepsilon}
\numberwithin{equation}{section}
\begin{document}

\title[Pre-Schwarzian norm estimate]{Pre-Schwarzian norm estimation for functions in the Ma-Minda-type starlike and convex classes}
\author[R. Biswas]{Raju Biswas}
\date{}
\address{Raju Biswas, Department of Mathematics, Raiganj University, Raiganj, West Bengal-733134, India.}
\email{rajubiswasjanu02@gmail.com}
\maketitle
\let\thefootnote\relax
\begin{abstract}
In this paper, we establish the sharp estimates of the pre-Schwarzian norm of functions $f$ in the Ma-Minda type starlike and convex classes
$\mathcal{S}^*(\varphi)$ and $\mathcal{C}(\varphi)$, respectively, whenever $\varphi(z)=3/\left(3+(\alpha-3)z-\alpha z^2\right)$ with $-3<\alpha\leq 1$, $\varphi(z)=(1+z)(1-s z)$ with $-1/3\leq s\leq 1/3$ and $\varphi(z)=1+z/\left((1-z) (1+\alpha z)\right)$ with $0\leq \alpha\leq 1/2$.
\end{abstract}
\section{Introduction}
\noindent Let $\mathcal{H}$ denote the class of all analytic functions in the unit disk $\mathbb{D}:=\{z\in\mathbb{C}:|z|<1\}$ and let $\mathcal{A}$ denote the subclass of $\mathcal{H}$ consisting of functions $f$ normalized by $f(0)=f'(0)-1=0$. 
Let $\mathcal{S}=\{f\in\mathcal{A}: \text{$f$ is univalent in $\mathbb{D}$}\}$. 
A domain $\Omega$ is said to be starlike with respect to a point $z_0\in\Omega$ if the line segment that connects $z_0$ to any other point in $\Omega$ is entirely contained within $\Omega$.
Specifically, if $z_0=0$, then $\Omega$ is simply called starlike. A function 
$f\in\mathcal{A}$ is said to be starlike if $f(\mathbb{D})$ is starlike with respect to the origin. 
Let $\mathcal{S}^*$ denote the class of all starlike functions in $\mathbb{D}$. It is well-known that a function $f\in\mathcal{A}$ is in $\mathcal{S}^*$ if, and only if, $\text{Re} (zf'(z)/f(z))>0$ for $z\in\mathbb{D}$. A domain $\Omega$ is called convex if it is starlike with respect to any point 
in $\Omega$. A function $f\in\mathcal{A}$ is said to be convex if $f(\mathbb{D})$ is convex. Let  $\mathcal{C}$ denote the class of all convex functions in $\mathbb{D}$.
It is well-known that a function $f\in\mathcal{A}$ is in $\mathcal{C}$ if, and only if, $\text{Re}\left(1+zf''(z)/f'(z)\right)>0$ for $z\in\mathbb{D}$. For more in-depth information regarding these classes, we refer to \cite{D1983, G1983, TTA2018}.\\[2mm]
\indent Let $\mathcal{B}$ be the class of all analytic functions $\omega:\mathbb{D}\rightarrow\mathbb{D}$ and $\mathcal{B}_0=\{\omega\in\mathcal{B} : \omega(0)=0\}$. 
Functions in $\mathcal{B}_0$ are called Schwarz function. According to Schwarz's lemma, if $\omega\in\mathcal{B}_0$, then $|\omega(z)|\leq |z|$ and $|\omega'(0)|\leq 1$. 
Strict inequality holds in both estimates unless $\omega(z)=e^{i\theta}z$, $\theta\in\mathbb{R}$. A sharpened form of the Schwarz lemma, known as the Schwarz-Pick lemma, 
gives the estimate  $|\omega'(z)|\leq (1-|\omega(z)|^2)/(1-|z|^2)$ for $z\in\mathbb{D}$ and $\omega\in\mathcal{B}$.\\[2mm]
\indent An analytic function $f$ in $\mathbb{D}$ is said to be subordinate to an analytic function $g$ in $\mathbb{D}$, written as $f\prec g$, if there exists a function $\omega\in\mathcal{B}_0$ such that $f(z)=g(\omega(z))$ for $z\in\mathbb{D}$. Moreover, if $g$ is univalent in $\mathbb{D}$,
then $f \prec g$ if, and only if, $f(0)=g(0)$ and $f(\mathbb{D})\subseteq g(\mathbb{D})$. For basic details and results on subordination classes, we refer
to \cite[Chapter 6]{D1983}. 
Ma and Minda \cite{MM1992} have introduced more general subclasses of starlike and convex functions as follows:
\beas\mathcal{S}^*(\varphi)=\left\{f\in\mathcal{S}:\frac{zf'(z)}{f(z)}\prec\varphi(z)\right\}\quad\text{and}\quad\mathcal{C}(\varphi)=\left\{f\in\mathcal{S}:1+\frac{zf''(z)}{f'(z)}\prec\varphi(z)\right\},\eeas
where the function $\varphi :\mathbb{D}\to\mathbb{ C}$, called Ma-Minda function, is analytic and univalent in $\mathbb{D}$ such that $\varphi(\mathbb{D})$ has positive real 
part, symmetric with respect to the real axis, starlike with respect to $\varphi(0)=1$ and $\varphi'(0)>0$. A Ma-Minda function has the Taylor series expansion 
of the form $\varphi(z)=1+\sum_{n=1}^\infty a_nz^n $ $(a_1>0)$.
We call $\mathcal{S}^*(\varphi)$ and $\mathcal{C}(\varphi)$ the Ma-Minda type starlike and Ma-Minda type convex classes associated with $\varphi$, respectively.  Evidently, $\mathcal{S}^*(\varphi)\subset\mathcal{S}^*$ and $\mathcal{C}(\varphi)\subset\mathcal{C}$ for every such $\varphi$. It is known that $f\in\mathcal{C}(\varphi)$ if, and only if, $zf'\in\mathcal{S}^*(\varphi)$ (see \cite{MM1992}).\\[2mm]
\indent It is evident that different choices of the function $\varphi$ in the classes $\mathcal{S}^*(\varphi)$ and $\mathcal{C}(\varphi)$ lead to the generation of several significant subclasses of $\mathcal{S}^*$ and $\mathcal{C}$, respectively. 
If, we choose $\varphi(z)=(1+z)/(1-z)$, then $\mathcal{S}^*(\varphi)=\mathcal{S}^*$ and $\mathcal{C}(\varphi)=\mathcal{C}$. For $\varphi(z)=(1+(1-2\alpha)z)/(1-z),~0\leq \alpha<1$, we get the classes $\mathcal{S}^*(\alpha)$ of starlike function of 
order $\alpha$ and $\mathcal{C}(\alpha)$ of convex function of order $\alpha$. If $\varphi=((1+z)/(1-z))^\alpha$ for $0< \alpha\leq 1$, then 
$\mathcal{S}^*(\varphi)=\mathcal{S}\mathcal{S}^*(\alpha)$ the class of strongly starlike function of order $\alpha$ and 
$\mathcal{C}(\varphi)=\mathcal{S}\mathcal{C}(\alpha)$ the class of strongly convex function of order $\alpha$ (see \cite{S1966}). Also for  $\varphi=(1+Az)/(1+Bz),~-1\leq B<A\leq 1$, we have the classes 
of Janowski starlike functions $\mathcal{S}^*[A,B]$ and Janowski convex functions $\mathcal{C}[A,B]$ (see \cite{J1973}).
For $\varphi(z)=(1+2/\pi^2(\log(1-\sqrt{z})/(1+\sqrt{z}))^2)$ the class $\mathcal{C}(\varphi)$ (resp., $\mathcal{S}^*(\varphi)$) is the class {\it UCV} (resp. {\it UST} ) of 
normalized uniformly convex (resp. starlike) functions (see \cite{R1993,1G1991,2G1991, R1991}). Ma and Minda \cite{1MM1992, MM1993} have studied the class {\it UCV} 
extensively.\\[2mm]
\indent In this paper, we consider three different classes of functions, such as $\mathcal{S}^*_{con}=\mathcal{S}^*(\varphi)$ with $\varphi(z)=3/\left(3+(\alpha-3)z-\alpha z^2\right)$, $-3<\alpha\leq 1$, $\mathcal{S}^*_{lim}=\mathcal{S}^*(\varphi)$ with $\varphi(z)=(1+z)(1-s z)$, $-1/3\leq s\leq 1/3$ and $\mathcal{S}^*_{cs}=\mathcal{S}^*(\varphi)$ with $\varphi(z)=1+z/\left((1-z) (1+\alpha z)\right)$, $0\leq \alpha\leq 1/2$. More specifically,
\beas 
&&\mathcal{S}^*_{con}=\left\{f\in\mathcal{A} :\frac{zf'(z)}{f(z)}\prec  \frac{3}{3+(\alpha-3)z-\alpha z^2},~-3<\alpha\leq 1\right\},\\
&&\mathcal{S}^*_{lim}=\left\{f\in\mathcal{A} :\frac{zf'(z)}{f(z)}\prec (1+z)(1-s z),~-1/3\leq s\leq 1/3\right\}\\\text{and}
&&\mathcal{S}^*_{cs}=\left\{f\in\mathcal{A} :\frac{zf'(z)}{f(z)}\prec 1+\frac{z}{(1-z) (1+\alpha z)},~0\leq \alpha\leq 1/2\right\}.
\eeas
The function $\varphi(z)=3/(3+(\alpha-3)z-\alpha z^2)$ with $-3<\alpha\leq 1$ maps
the unit disk $\mathbb{D}$ onto a domain, which is geometrically similar to a conchoid with no loops, as illustrated in Figure \ref{Fig1}.
Moreover, $\varphi(\mathbb{D})$ is symmetric respecting the real axis, $\varphi$ is starlike with
respect to $\varphi(0)=1$. It is evident that $\varphi'(0)=(3-\alpha)/3>0$ and $\varphi$ has positive real part in $\mathbb{D}$. Thus, $\varphi$ satisfies the category of Ma-Minda 
functions. Furthermore, a function $f\in\mathcal{A}$ is in $\mathcal{S}^*_{con}$, $\alpha\in(-3,1]$ if, and only if, for $z\in\mathbb{D}$, the quantity $zf'(z)/f(z)$ takes all values 
on the right hand side of the curve  
\beas (u-a)\left(u^2+v^2\right)-k\left(u-\frac{1}{2}\right)^2=0,~\text{where}~a=\frac{9(1+\alpha)}{2(3+\alpha)^2}~\text{and}~k=\frac{54}{(3+\alpha)^2(3-\alpha)}.\eeas
A function $f\in \mathcal{S}^*_{con}$ if, and only if, there exists an analytic function $p$ with $p(0)=1$ and $p(z)\prec 3/(3+(\alpha-3)z-\alpha z^2)$ in $\mathbb{D}$ such that
\bea\label{e1} f(z)=z\;\mathrm{exp}\left(\int_0^z \frac{p(t)-1}{t}dt\right).\eea
If we choose $p(z)=3/(3+(\alpha-3)z-\alpha z^2)$ in \eqref{e1}, we obtain
\beas f_{\alpha,1}(z)=z\;\mathrm{exp}\left(\int_0^z \frac{3/(3+(\alpha-3)t-\alpha t^2)-1}{t}dt\right)\in\mathcal{S}^*_{con}.\eeas
\begin{figure}[H]
\begin{minipage}[c]{0.49\linewidth}
\centering
\includegraphics[scale=0.7]{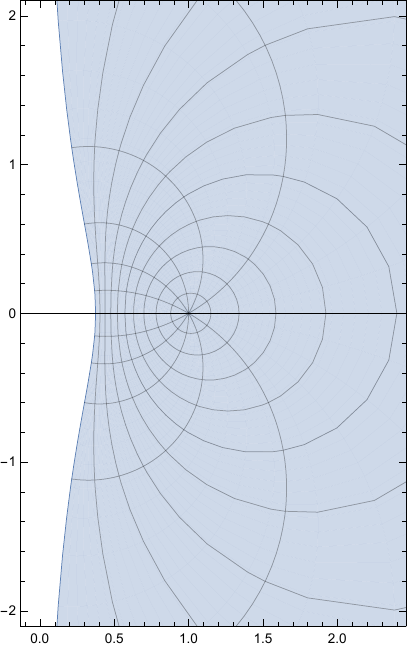}
\caption{Image of $\mathbb{D}$ under the mapping $3/\left(3+(\alpha-3)z-\alpha z^2\right)$, $\alpha=-1$}
\label{Fig1}
\end{minipage}
\begin{minipage}[c]{0.5\linewidth}
\centering
\includegraphics[scale=0.7]{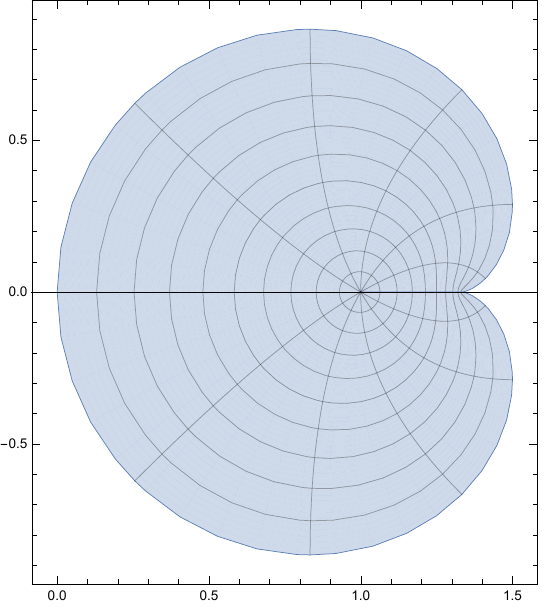}
\caption{Image of $\mathbb{D}$ under the mapping $(1+z)(1-s z)$, $s=1/3$}
\label{Fig4}
\end{minipage}
\end{figure}
\noindent The function $\varphi(z)=(1+z)(1-s z)$, $-1/3\leq s\leq 1/3$ maps
the unit disk $\mathbb{D}$ onto a Lima\c{c}on domain given by
\beas \mathcal{L}(s):=\left\{u+iv\in\mathbb{C}: \left((u-1)^2+v^2-s^2\right)^2<(1-s)^2\left((u-1-s)^2+v^2\right)\right\},\eeas 
as illustrated in Figure \ref{Fig4}.
Moreover, $\varphi(\mathbb{D})$ is symmetric respecting the real axis, $\varphi$ is starlike in the unit disk with
respect to $\varphi(0)=1$. It is evident that $\varphi'(0)=1-s>0$ and $\varphi$ has positive real part in $\mathbb{D}$. Thus, $\varphi$ satisfies the category of Ma-Minda 
functions. Furthermore, a function $f\in\mathcal{A}$ is in $\mathcal{S}^*_{lim}$, $s\in[-1/3,1/3]$ if, and only if, there exists an analytic function $p$ with $p(0)=1$ and $p(z)\prec (1+z)(1-s z)$ in $\mathbb{D}$ such that
\bea\label{ee1} f(z)=z\;\mathrm{exp}\left(\int_0^z \frac{p(t)-1}{t}dt\right).\eea
If we choose $p(z)=(1+z)(1-sz)$ in \eqref{ee1}, we obtain a function given by
\beas f_{s}(z)=z\;\mathrm{exp}\left((1-s)z-\frac{s}{2}z^2\right)\in\mathcal{S}^*_{lim}.\eeas
\noindent The function
\beas \varphi(z)=1+\frac{z}{(1-z) (1+\alpha z)}=1+\frac{1}{1+\alpha} \sum_{n=1}^\infty \left(1-(-1)^{n}\alpha^{n}\right)z^n,~\alpha\in[0,1/2]\eeas
 maps the unit disk $\mathbb{D}$ onto a onto domain bounded by cissoid of Diocles $\mathbb{CS}(\alpha)$ defined by
\beas \mathbb{CS}(\alpha):=\left\{u+i v\in\mathbb{C}: \left(u-\frac{2\alpha -1}{2(\alpha-1)}\right)\left((u-1)^2+v^2\right)+\frac{2\alpha}{(1+\alpha)^2(\alpha-1)}v^2=0\right\}\eeas
for some $0\leq \alpha\leq 1/2$, as illustrated in Figure \ref{Fig2}. Moreover, $\varphi(\mathbb{D})$ is symmetric about the real axis, $\varphi$ is starlike with
respect to the point $\varphi(0)=1$, $\varphi$ has positive real part in $\mathbb{D}$ and satisfies $\varphi'(0)=1>0$. Thus, $\varphi(z)$ satisfies the category of Ma-Minda functions.
\begin{figure}
\centering
\includegraphics[scale=0.7]{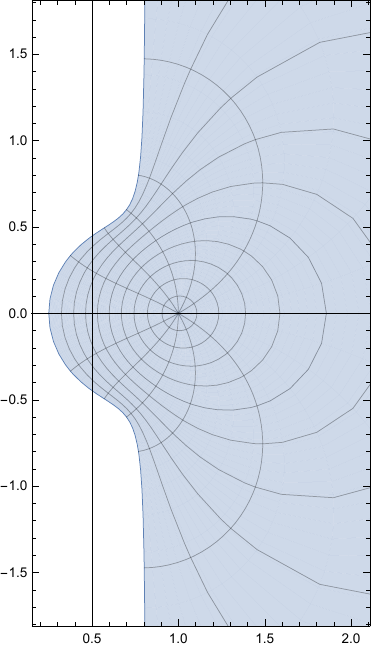}
\caption{Image of $\mathbb{D}$ under the mapping $1+z/\left((1-z) (1+\alpha z)\right)$, $\alpha=1/3$}
\label{Fig2}
\end{figure}
A function $f\in \mathcal{S}^*_{cs}$ if, and only if, there exists 
an analytic function $p$ with $p(0)=1$ and $p(z)\prec 1+z/((1-z) (1+\alpha z))$, $\alpha\in[0, 1/2]$ in $\mathbb{D}$ such that
\bea\label{e3} f(z)=z\;\mathrm{exp}\left(\int_0^z \frac{p(t)-1}{t}dt\right).\eea 
If we choose $p(z)=1+z/((1-z) (1+\alpha z))$ in \eqref{e3}, we obtain
\beas f_{\alpha,2}(z)=z\exp\left(\int_0^z \frac{dt}{(1-t) (1+\alpha t)}\right)=z+z^2+\frac{2-\alpha}{2} z^3+\frac{2 \alpha^2-5 \alpha+6}{6}  z^4+\cdots\in\mathcal{S}^*_{cs}.\eeas
The functions $f_{\alpha, 1}(z)$, $f_s(z)$ and $f_{\alpha, 2}(z)$ plays the role of extremal function for many extremal problems in the classes $\mathcal{S}^*_{con}$, 
$\mathcal{S}^*_{lim}$ and $\mathcal{S}^*_{cs}$, respectively. \\[2mm]
\indent
It is well-known that $f\in\mathcal{C}_{con}$ (resp. $\mathcal{C}_{lim}$ and $\mathcal{C}_{cs}$) if, and only if, $zf'\in\mathcal{S}^*_{con}$ (resp. $\mathcal{S}^*_{lim}$ and $\mathcal{S}^*_{cs}$), where  the classes $\mathcal{C}_{con}$, $\mathcal{C}_{lim}$ and $\mathcal{C}_{cs}$ are defined by
\beas
&&\mathcal{C}_{con}=\left\{f\in\mathcal{A}:1+\frac{zf''(z)}{f'(z)}\prec \frac{3}{3+(\alpha-3)z-\alpha z^2},~-3<\alpha\leq 1\right\},\\
&&\mathcal{C}_{lim}=\left\{f\in\mathcal{A}:1+\frac{zf''(z)}{f'(z)}\prec (1+z)(1-s z),~-1/3\leq s\leq 1/3\right\}\eeas
\beas\text{and}\quad\mathcal{C}_{cs}=\left\{f\in\mathcal{A}:1+\frac{zf''(z)}{f'(z)}\prec 1+\frac{z}{(1-z) (1+\alpha z)},~0\leq \alpha\leq 1/2\right\}.\eeas
For a more in-depth analysis of these classes, we refer to \cite{S2011, MEY2019, MN2024}.
 \section{Pre-Schwarzian Norm}
An analytic function $f(z)$ in a domain $\Omega$ is said to be locally univalent if for each $z_0\in\Omega$, there exists a neighborhood $U$ of $z_0$ such that $f(z)$ is univalent in $U$. The Jacobian of a complex-valued function $f(z)=u(x, y)+iv(x, y)$ defined by $J_f(z)=u_x v_y-u_y v_x=|f_z(z)|^2-|f_{\ol{z}}(z)|^2$, provided all the partial derivatives exist. If $f$ is analytic on a simply connected domain $\Omega$, the Jacobian takes the form $J_f(z)=|f'(z)|^2$.
It is well-known that an analytic function $f$ is locally univalent at $z_0$ if, and only if, $J_f(z_0)\not=0$ (see \cite[Chapter 1]{D1983}).
Let $\mathcal{LU}:=\{f\in\mathcal{H}:f'(z)\not= 0~~\text{ for all }~~z\in\mathbb{D}\}$. For $f\in\mathcal{LU}$, the pre-Schwarzian derivative is defined by
\beas P_f(z):=\frac{f''(z)}{f'(z)},\eeas
and the pre-Schwarzian norm (the hyperbolic sup-norm) is defined by
\beas \Vert P_f\Vert :=\sup_{z\in\mathbb{D}}\;(1-|z|^2)\left|P_f(z)\right|.\eeas
Pre-Schwarzian norm plays an important rule in the geometric function theory and Teichm\"{u}ller theory. It is well-known that if $f$ is a univalent function in $\mathbb{D}$, then 
$\Vert P_f \Vert\leq 6$. This equality is attained for the Koebe function or its rotation. One of the most widely used univalence criteria for locally univalent analytic functions is Becker's univalence criteria \cite{B1972}.
It states that if $f\in\mathcal{LU}$ and $\sup_{z\in\mathbb{D}}\left(1-|z|^2\right) \left|zP_f(z)\right|\leq1$, then $f$ is univalent in $\mathbb{D}$.
Later, the sharpness of the constant $1$ is proved by Becker and Pommerenke \cite{BP1984}. In 1976, Yamashita \cite{Y1976} 
proved that $\Vert P_f \Vert<\infty $ is finite if, and only if, $f$ is uniformly 
locally univalent in $\mathbb{D}$. Moreover, if $\Vert P_f\Vert<2$, then $f$ is bounded in $\mathbb{D}$ (see \cite{KS2002}).\\[2mm]
\indent 
In the field of geometric function theory, numerous researchers have investigated the pre-Schwarzian norm for various subclasses of analytic and univalent functions.
In 1998, Sugawa \cite{S1998} established the sharp estimate of the pre-Schwarzian norm for strongly starlike functions of order $\alpha$ ($0<\alpha\leq 1$).
In $1999$, Yamashita \cite{Y1999} proved that $\Vert P_f\Vert \leq 6-4\alpha$ for $f\in\mathcal{S}^*(\alpha)$ and $\Vert P_f\Vert \leq 4(1-\alpha)$ for $f\in\mathcal{C}(\alpha)$, 
where $0\leq \alpha<1$ and both the estimates are sharp. In $2000$, Okuyama \cite{O2000} established the sharp estimate of the pre-Schwarzian norm for $\alpha$-spirallike functions. Kim and Sugawa \cite{KS2006} established the sharp 
estimate of the pre-Schwarzian norm for Janowski convex functions (see also \cite{PS2008}). Ponnusamy and Sahoo \cite{PS2010} 
obtained the sharp estimates of the pre-Schwarzian norm for functions in the class
$\mathcal{S}^*[\alpha,\beta]:=S^*\left(\left((1+(1-2\beta)z)/(1-z)\right)^\alpha\right)$, where $0<\alpha\leq 1$ and $0\leq \beta<1$. In $2014$, Aghalary and Orouji 
\cite{AO2014} obtained the sharp estimate of the pre-Schwarzian norm for $\alpha$-spirallike function of order $\rho\in[0,1)$, where $\alpha\in(-\pi/2,\pi/2)$. The pre-
Schwarzian norm of certain integral transform of $f$ for certain subclass of $f$ has been also studied in the literature. For a more in-depth investigation, we refer to 
\cite{KPS2004,PPS2008,PS2008, AP2023,1AP2023,2AP2023,AP2024,CKPS2005,ABM2025} and the references therein.\\[2mm]
\indent In this paper, we establish sharp estimates of the pre-Schwarzian norms for functions in the Ma-Minda type starlike 
classes $\mathcal{S}^*_{con}$, $\mathcal{S}^*_{lim}$ and $\mathcal{S}^*_{cs}$, and in the Ma-Minda type convex classes $\mathcal{C}_{con}$, $\mathcal{C}_{lim}$ and $\mathcal{C}_{cs}$.
\section{Main results}
\noindent In the following result, we establish the sharp estimate of the pre-Schwarzian norm for functions in the Ma-Minda type starlike class $\mathcal{S}^*_{con}$.
\begin{theo}\label{Th2}
Let $f\in\mathcal{S}^*_{con}$. Then the pre-Schwarzian norm satisfies the following sharp inequality
\beas\Vert P_f\Vert \leq \frac{2( \alpha+6)}{3+\alpha}.\eeas
\end{theo}
\begin{proof}
If $f\in\mathcal{S}^*_{con}$, it follows from the definition of the class $\mathcal{S}^*_{con}$ that
\beas\frac{zf'(z)}{f(z)}\prec \frac{3}{3+(\alpha-3)z-\alpha z^2}.\eeas
Thus, there exists a Schwarz function $\omega(z)\in\mathcal{B}_0$ such that
\beas\frac{zf'(z)}{f(z)}=\frac{3}{3+(\alpha-3)\omega(z)-\alpha (\omega(z))^2}.\eeas
Taking logarithmic derivative on both sides with respect to $z$, we obtain
\beas P_f(z)=\frac{f''(z)}{f'(z)}=\frac{\left(3-\alpha+2\alpha \omega(z)\right)\omega'(z)}{3+(\alpha-3)\omega(z)-\alpha (\omega(z))^2}+\frac{1}{z}\left(\frac{3}{3+(\alpha-3)\omega(z)-\alpha (\omega(z))^2}-1\right).\eeas
Note that the function 
\be\label{g1}\frac{3}{3+(\alpha-3)z-\alpha z^2}=\frac{3}{\alpha+3} \left(\frac{\alpha}{\alpha z+3}+\frac{1}{1-z}\right)=\frac{3}{\alpha+3} \sum_{n=0}^\infty \left(1+(-1)^{n}\left(\frac{\alpha}{3}\right)^{n+1}\right)z^n\ee
has the positive Taylor coefficients about $z = 0$ for $-3<\alpha\leq 1$.
In view of Schwarz-Pick lemma, we have 
\beas
(1-|z|^2)|P_f(z)|&=& (1-|z|^2)\left|\frac{\left(3-\alpha+2\alpha \omega(z)\right)\omega'(z)}{3+(\alpha-3)\omega(z)-\alpha (\omega(z))^2}\right.\\[2mm]
&&\left.+\frac{1}{z}\left(\frac{3}{3+(\alpha-3)\omega(z)-\alpha (\omega(z))^2}-1\right)\right|\\
&\leq&\frac{(1-|z|^2)\left(3-\alpha+2\alpha |\omega(z)|\right)|\omega'(z)|}{3+(\alpha-3)|\omega(z)|-\alpha |\omega(z)|^2}\\[2mm]
&&+\frac{(1-|z|^2)}{|z|}\left(\frac{3}{3+(\alpha-3)|\omega(z)|-\alpha |\omega(z)|^2}-1\right)\\[2mm]
&\leq&\frac{\left(3-\alpha+2\alpha |\omega(z)|\right)\left(1-|\omega(z)|^2\right)}{3+(\alpha-3)|\omega(z)|-\alpha |\omega(z)|^2}\\[2mm]
&&+\frac{(1-|z|^2)}{|z|}\left(\frac{3}{3+(\alpha-3)|\omega(z)|-\alpha |\omega(z)|^2}-1\right).\eeas
For $0<t:=|\omega(z)|\leq |z|<1$, we have
\beas(1-|z|^2)|P_f(z)|\leq \frac{\left(3-\alpha+2\alpha t\right)\left(1-t^2\right)}{3+(\alpha-3)t-\alpha t^2}+\frac{(1-|z|^2)}{|z|}\left(\frac{3}{3+(\alpha-3)t-\alpha t^2}-1\right).\eeas
Therefore, the pre-Schwarzian norm for functions $f$ in the class $\mathcal{S}^*_{con}$ is
\bea\label{e9}\Vert P_f\Vert=\sup_{z\in\mathbb{D}}\;(1-|z|^2)|P_f(z)|\leq \sup_{0< t\leq|z|<1}G_1(|z|, t),\eea
where 
\beas G_1(r, t)=\frac{\left(3-\alpha+2\alpha t\right)\left(1-t^2\right)}{3+(\alpha-3)t-\alpha t^2}+\frac{(1-r^2)}{r}\left(\frac{3}{3+(\alpha-3)t-\alpha t^2}-1\right)~ \text{for}\quad |z|=r.\eeas
The objective is to ascertain the supremum of $G_1(r, t)$ over $\Omega=\{(r, t): 0< t\leq r<1\}$.
Differentiating partially $G_1(r, t)$ with respect to $r$, we obtain
\beas\frac{\pa}{\pa r}G_1(r, t)=-\left(\frac{1}{r^2}+1\right)\left(\frac{3}{3+(\alpha-3)t-\alpha t^2}-1\right)<0,\eeas
which shows that $G_1(r, t)$ is a monotonically decreasing function of $r\in[t,1)$. Hence, we have $G_1(r, t) \leq G_1(t, t)=G_2(t)$,
where
\bea\label{e10} G_2(t)
&=&\frac{(t+1) (6-2\alpha+3 \alpha t)}{3+\alpha t},\quad t\in(0,1).\eea
Now we consider the following cases.\\
{\bf Case 1.} If $\alpha=0$, then from (\ref{e10}), we have $G_2(t)=2(1+t)$. Hence, we have $\Vert P_f\Vert\leq \lim_{t\to1^-}G_2(t)=4$.\\
 {\bf Case 2.} If $\alpha\not=0$.
Differentiate $G_2(t)$ with respect to $t$, we obtain
\beas G_2'(t)&=&\frac{3 \alpha^2 t^2+18 \alpha t+2 \alpha^2-3 \alpha+18}{(3+\alpha t)^2}.\eeas
The roots of the equation $G_2'(t)=0$ are
\beas t_1=\frac{-9 \alpha-\sqrt{3} \sqrt{9 \alpha^2+3 \alpha^3-2 \alpha^4}}{3 \alpha^2}\quad\text{and}\quad t_2=\frac{-9 \alpha+\sqrt{3} \sqrt{9 \alpha^2+3 \alpha^3-2 \alpha^4}}{3 \alpha^2}.\eeas
It is evident that $9 \alpha^2+3 \alpha^3-2 \alpha^4\geq 0$ for $-3/2\leq \alpha\leq 1$.\\
{\bf Sub-case 2.1.} If $-3<\alpha <-3/2$, then the equation $G_2'(t)=0$ has no real root and it follows that $G_2'(1/2)=(72 + 24 \alpha + 11 \alpha^2)/(6 + \alpha)^2>0$ for $-3<\alpha <-3/2$. Hence, $G_2'(t)>0$ for $t\in(0,1)$, which shows that $G_2(t)$ is a monotonically increasing function of $t$. From (\ref{e9}), we have 
\beas \Vert P_f\Vert\leq \lim_{t\to 1^-}G_2(t)=\frac{2( \alpha+6)}{3+\alpha}.\eeas
{\bf Sub-case 2.2.} Let $-3/2\leq \alpha<0$.  Then, the inequality $-9 \alpha-\sqrt{3} \sqrt{9 \alpha^2+3 \alpha^3-2 \alpha^4}>0$ is equivalent to $3 \alpha^2 (18 - 3 \alpha + 2 \alpha^2)>0$, which is true for $\alpha\in[-3/2,0)$. Therefore, both the roots $t_1$ and $t_2$ are non-negative. We claim that, both $t_1$ and $t_2$ are greater than $1$. Note that 
$t_1>1$ is equivalent to
\beas H_1(\alpha):=108 \alpha^2+9 \alpha^3-15 \alpha^4+18 \sqrt{3}\alpha \sqrt{9 \alpha^2+3 \alpha^3-2 \alpha^4}>0,\eeas
which is true for $\alpha\in[-3/2,0)$, as illustrated in Figure \ref{Fig3}. Similarly, $t_2>1$ is equivalent to
\beas H_2(\alpha):=108 \alpha^2+9 \alpha^3-15 \alpha^4-18 \sqrt{3}\alpha \sqrt{9 \alpha^2+3 \alpha^3-2 \alpha^4}>0,\eeas
which is true for $\alpha\in[-3/2,0)$, as illustrated in Figure \ref{Fig3}.
\begin{figure}[H]
\centering
\includegraphics[scale=0.9]{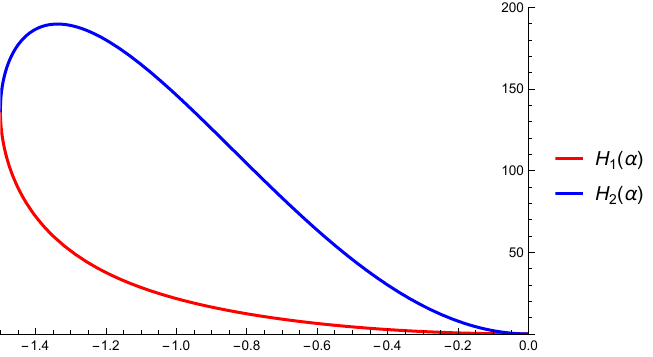}
\caption{Graph of $H_1(\alpha)$ and $H_2(\alpha)$ for $\alpha$ in $[-3/2,0)$}
\label{Fig3}
\end{figure}
\noindent Therefore, the equation $G_2'(t)=0$ has no positive root in $(0, 1)$ and it follows that $G_2'(1/2)=(72 + 24 \alpha + 11 \alpha^2)/(6 + \alpha)^2>0$ for $-3/2\leq \alpha <0$. Hence, $G_2'(t)>0$ for $t\in(0,1)$, which shows that $G_2(t)$ is a monotonically increasing function of $t$. From (\ref{e9}), we have 
\beas \Vert P_f\Vert\leq \lim_{t\to 1^-}G_2(t)=\frac{2( \alpha+6)}{3+\alpha}.\eeas
{\bf Sub-case 2.3.} Let $0< \alpha\leq 1$.  It is evident that $t_1<0$. We claim that $t_2<0$. Note that 
$t_2<0$ is equivalent to
\beas -3 \left(2\alpha^4-3\alpha^3+18\alpha^2\right)<0,\eeas
which is true for $\alpha\in(0, 1]$. Therefore, the equation $G_2'(t)=0$ has no positive root in $(0, 1)$ and it follows that $G_2'(1/2)=(72 + 24 \alpha + 11 \alpha^2)/(6 + \alpha)^2>0$ for $0<\alpha \leq 1$. Hence, $G_2'(t)>0$ for $t\in(0,1)$, which shows that $G_2(t)$ is a monotonically increasing function of $t$. From (\ref{e9}), we have 
\beas \Vert P_f\Vert\leq \lim_{t\to 1^-}G_2(t)=\frac{2( \alpha+6)}{3+\alpha}.\eeas
From Case 1 and Case 2, we have 
\beas\Vert P_f\Vert\leq 
\dfrac{2( \alpha+6)}{3+\alpha}.\eeas
\indent To show that the estimate is sharp, we consider the function $f_1$ given by
\beas f_1(z)=z\;\mathrm{exp}\left(\int_0^z\frac {3 - \alpha + \alpha t} {(1 - t) (3 + \alpha t)} dt\right).\eeas
Differentiate $f_1(z)$ twice with respect to $z$, we obtain
\beas f_1'(z)&=&\frac{3}{(1-z) (3+\alpha z)}\;\mathrm{exp}\left(\int_0^z\frac {3 - \alpha + \alpha t} {(1 - t) (3 + \alpha t)} dt\right)\quad\text{and}\\[2mm]
 f_1''(z)&=&\frac{3 (6-2\alpha+3\alpha z)}{(1-z)^2 (3+\alpha z)^2}\;\mathrm{exp}\left(\int_0^z\frac {3 - \alpha + \alpha t} {(1 - t) (3 + \alpha t)} dt\right)\eeas
The pre-Schwarzian norm of $f_1$ is given by
\beas \Vert P_{f_1}\Vert =\sup_{z\in\mathbb{D}}\;(1-|z|^2)|P_{f_1}(z)|=\sup_{z\in\mathbb{D}}\;(1-|z|^2)\left|\frac{(6-2\alpha+3\alpha z)}{(1-z) (3+\alpha z)}\right|.\eeas
On the positive real axis, we note that
\beas
\sup_{0\leq r<1}\left((1-r^2)\frac{(6-2\alpha+3\alpha r)}{(1-r) (3+\alpha r)}\right)=
\dfrac{2( \alpha+6)}{3+\alpha}\quad\text{for}\quad \alpha\in(-3, 1].\eeas
This completes the proof.
\end{proof}
In the following result, we establish the sharp estimate of the pre-Schwarzian norm for functions $f$ in the Ma-Minda type starlike class $\mathcal{S}^*_{lim}$.
\begin{theo} Let $f\in\mathcal{S}^*_{lim}$. Then the pre-Schwarzian norm satisfies the following sharp inequality
\beas\Vert P_f\Vert \leq \frac{(1+t_s)\left(1-s+2s t_s\right)}{1+s t_s}+(1-t_s^2)\left(1-s+s t_s\right),\eeas
where $t_s\in(0, 1)$ is the unique root of the equation 
\beas -3 s^3 t^4+\left(2 s^3-8 s^2\right) t^3+\left(s^3+6 s^2-7 s\right) t^2+\left(2 s^2+6 s-2\right) t+s^2+s+1=0.\eeas
\end{theo}
\begin{proof}
If $f\in\mathcal{S}^*_{lim}$, it follows from the definition of the class $\mathcal{S}^*_{lim}$ that
\beas\frac{zf'(z)}{f(z)}\prec (1-s z)(1+z).\eeas
Thus, there exists a Schwarz function $\omega\in\mathcal{B}_0$ such that
\beas\frac{zf'(z)}{f(z)}=(1-s\omega(z))(1+\omega(z)).\eeas
Taking logarithmic derivative on both sides with respect to $z$, we obtain
\beas
\frac{f''(z)}{f'(z)}
&=&\frac{(1-s-2s\omega(z))\omega'(z)}{(1+\omega(z))(1-s\omega(z))}+\frac{1}{z}\left((1-s\omega(z))(1+\omega(z))-1\right).\eeas
Note that $|\omega(z)|\leq |z|<1$ and the Taylor  series expansion of the function 
\beas \frac{1}{(1+z) (1-s z)}
&=&\sum_{n=0}^\infty \frac{\left((-1)^n+s^{n+1}\right) }{1+s}z^n\\
&=&1-(1-s) z+\left(1-s+s^2\right) z^2-\left(1-s+s^2-s^3\right) z^3+O\left(z^4\right).\eeas
Hence, we have  
\beas \left|\frac{1}{(1+z) (1-s z)}\right|&\leq &1+(1-s)|z|+\left(1-s+s^2\right) |z|^2+\left(1-s+s^2-s^3\right) |z|^3\\
&&+O\left(|z|^4\right)\\
&=&\frac{1}{(1-|z|) (1+s |z|)}.\eeas
In view of  Schwarz-Pick lemma, we have
\beas (1-|z|^2)|P_f(z)|&=&(1-|z|^2)\left|\frac{(1-s-2s\omega(z))\omega'(z)}{(1+\omega(z))(1-s\omega(z))}+\frac{(1-s\omega(z))(1+\omega(z))-1}{z}\right|\\[2mm]
&\leq& (1-|z|^2)\left(\frac{|\omega'(z)|\left(1-s+2s|\omega(z)|\right)}{(1-|\omega(z)|)(1+s|\omega(z)|)}+\frac{(1-s)|\omega(z)|+s|\omega(z)|^2}{|z|}\right)\\[2mm]
&\leq& \frac{(1+|\omega(z)|)\left(1-s+2s|\omega(z)|\right)}{(1+s|\omega(z)|)}+\frac{(1-|z|^2)\left((1-s)|\omega(z)|+s|\omega(z)|^2\right)}{|z|}.\eeas
For $0< t:=|\omega(z)|\leq |z|<1$, we have
\beas (1-|z|^2)|P_f(z)|\leq \frac{(1+t)\left(1-s+2s t\right)}{(1+st)}+\frac{(1-|z|^2)\left((1-s)t+st^2\right)}{|z|}.\eeas
Therefore, the pre-Schwarzian norm for functions $f$ in the class $\mathcal{S}^*_{lim}$ is
\bea \label{g4}\Vert P_f\Vert=\sup_{z\in\mathbb{D}}\; (1-|z|^2)|P_f(z)|\leq \sup_{0< t\leq |z|<1} G_3(|z|,t),\eea
where 
\beas G_3(r, t)=\frac{(1+t)\left(1-s+2s t\right)}{(1+st)}+\frac{(1-r^2)\left((1-s)t+st^2\right)}{r}\quad\text{for}\quad |z|=r\in(0, 1).\eeas
The objective is to determine the supremum of $G_3(r, t)$ on $\Omega=\{(r, t): 0< t\leq r<1\}$.
Differentiate $G_3(r, t)$ partially with respect to $r$, we obtain
\beas\frac{\pa}{\pa r}G_3(r, t)=-\left(\frac{1}{r^2}+1\right)\left((1-s)t+st^2\right)<0,\eeas
which shows that $G_3(r, t)$ is a monotonically decreasing function of $r\in[t,1)$. Hence, we have $G_3(r, t) \leq G_3(t, t)=G_4(t)$,
where
\beas G_4(t)=\frac{(1+t)\left(1-s+2s t\right)}{(1+st)}+(1-t^2)\left(1-s+st\right),\quad t\in(0,1).\eeas
Differentiate $G_4(t)$ twice with respect to $t$, we obtain
\beas G_4'(t)&=&\frac{-3 s^3 t^4+\left(2 s^3-8 s^2\right) t^3+\left(s^3+6 s^2-7 s\right) t^2+\left(2 s^2+6 s-2\right) t+s^2+s+1}{(1+s t)^2}\\[2mm]
G_4''(t)&=&\frac{-2\left(3 s^4 t^4+\left(10 s^3-s^4\right) t^3+\left(12 s^2-3 s^3\right) t^2+\left(6 s-3 s^2\right) t+s^3-2 s+1\right)}{(1+s t)^3}.\eeas
Therefore, we have $G_4''(t)<0$ for $t\in(0, 1)$ and $s\in[-1/3,1/3]$. Therefore, $G_4'(t)$ is a monotonically decreasing function of $t\in(0,1)$ with $\lim_{t\to0^+}G_4'(t)=1 + s + s^2>0$ and $\lim_{t\to1^-}G_4'(t)=(-1 + s^2)/(1 + s)^2<0$. This leads to the conclusion that the equation $G_4'(t)=0$ has
the unique root $t_s$ in $(0,1)$. This shows that $G_4(t)$ attains its maximum at $t_s$. From (\ref{g4}), we have
\beas\Vert P_f\Vert \leq G_4(t_s)=\frac{(1+t_s)\left(1-s+2s t_s\right)}{1+s t_s}+(1-t_s^2)\left(1-s+s t_s\right),\eeas
where $t_s\in(0, 1)$ is the unique root of the equation 
\be\label{g2} H_s(t):=\frac{-3 s^3 t^4+\left(2 s^3-8 s^2\right) t^3+\left(s^3+6 s^2-7 s\right) t^2+\left(2 s^2+6 s-2\right) t+s^2+s+1}{(1+s t)^2}=0.\ee
\indent To show that the estimate is sharp, we consider the function $f_2$ given by
\beas f_2(z)=z\;\mathrm{exp}\left(\int_0^z\frac {(1-t) (1+s t)-1} {t} dt\right).\eeas
Differentiate $f_2(z)$ twice with respect to $z$, we obtain
\beas f_2'(z)&=&(1-z) (1+s z)\;\exp\left(\int_0^z\frac {(1-t) (1+s t)-1} {t} dt\right)\quad\text{and}\\[2mm]
 f_2''(z)&=&\frac{(1-z) (1+s z)\left((1-z) (1+s z)-1\right)}{z}\exp\left(\int_0^z\frac {(1-t) (1+s t)-1} {t} dt\right)\\[2mm]
 &&+(s-1-2sz)\exp\left(\int_0^z\frac {(1-t) (1+s t)-1} {t} dt\right).\eeas
Hence, we have 
\beas \frac{f_2''(z)}{f_2'(z)}=\frac{\left((1-z) (1+s z)-1\right)}{z}+\frac{(s-1-2sz)}{(1-z) (1+s z)}=-\frac{1-s+2sz}{(1-z) (1+s z)}-(1-s+s z)\eeas
The pre-Schwarzian norm of $f_2$ is given by
\beas \Vert P_{f_2}\Vert =\sup_{z\in\mathbb{D}}\;(1-|z|^2)|P_{f_2}(z)|=\sup_{z\in\mathbb{D}}\;(1-|z|^2)\left|\frac{1-s+2sz}{(1-z) (1+s z)}+(1-s+s z)\right|.\eeas
On the positive real axis, we note that
\beas
&&\sup_{0\leq r<1}\left((1-r^2)\left(\frac{1-s+2sr}{(1-r) (1+s r)}+(1-s+s r)\right)\right)\\
&=&\frac{(1+r_s)\left(1-s+2s r_s\right)}{1+s r_s}+(1-r_s^2)\left(1-s+s r_s\right),\eeas
where $r_s\in(0, 1)$ is the unique root of the equation 
\beas -3 s^3 r^4+\left(2 s^3-8 s^2\right) r^3+\left(s^3+6 s^2-7 s\right) r^2+\left(2 s^2+6 s-2\right) r+s^2+s+1=0.\eeas
This completes the proof.
\end{proof}
\noindent In Table \ref{tab1} and Figure \ref{Fig5}, we obtain the values of $t_s$ and $\Vert P_f\Vert$ for certain values of $s\in[-1/3, 1/3]$.
\begin{table}[H]
\centering
\begin{tabular}{*{7}{|c}|}
\hline
$s$				&-1/3    				&1/3    		&1/5		&     -1/5			&-1/4			&1/9\\
\hline
$t_s$				&0.251584			&0.827189	&0.721552	&0.319062	&0.289353	&0.628095\\
\hline
$\Vert P_f\Vert$&2.76279			&2.04225	&2.09044	&2.53147		&2.61535	&2.14568\\
\hline
\end{tabular}
\caption{$t_s$ is the unique positive root of the equation (\ref{g2}) in $(0,1)$}
\label{tab1}\end{table}
\begin{figure}[H]
\centering
\includegraphics[scale=0.8]{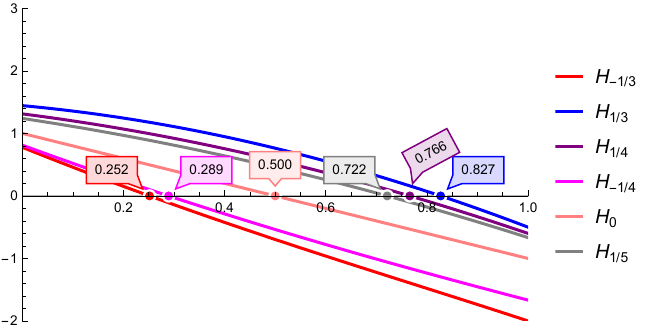}
\caption{Graph of $H_s(t)$ for different values of $s$ in $[-1/3, 1/3]$}
\label{Fig5}
\end{figure}
In the following result, we establish the pre-Schwarzian norm estimate for functions $f$ in the Ma-Minda type starlike class $\mathcal{S}^*_{cs}$.
\begin{theo}\label{Th1} Let $f\in\mathcal{S}^*_{cs}$. Then the pre-Schwarzian norm satisfies the following inequality
\beas\Vert P_f\Vert \leq \frac{2 (2-\alpha)}{(1+\alpha) (1-2 \alpha)}.\eeas
\end{theo}
\begin{proof}
If $f\in\mathcal{S}^*_{cs}$, it follows from the definition of the class $\mathcal{S}^*_{cs}$ that
\beas\frac{zf'(z)}{f(z)}\prec 1+\frac{z}{(1-z) (1+\alpha z)}.\eeas
Thus, there exists a Schwarz function $\omega\in\mathcal{B}_0$ such that
\beas\frac{zf'(z)}{f(z)}=1+\frac{\omega(z)}{(1-\omega(z)) (1+\alpha \omega(z))}=\frac{1+\alpha \omega(z)-\alpha(\omega(z))^2}{(1-\omega(z)) (1+\alpha \omega(z))}.\eeas
Taking logarithmic derivative on both sides with respect to $z$, we obtain
\beas
\frac{f''(z)}{f'(z)}
&=&\frac{\omega'(z)(1+\alpha(\omega(z))^2)}{(1+\alpha \omega(z)-\alpha(\omega(z))^2)(1-\omega(z)) (1+\alpha \omega(z))}+\frac{\omega(z)}{z(1-\omega(z)) (1+\alpha \omega(z))}.\eeas
Note that $|\omega(z)|\leq |z|<1$ and $|1+\alpha \omega(z)(1-\omega(z))|\geq 1-\alpha |\omega(z)||1-\omega(z)|\geq 1-\alpha |\omega(z)|(1+|\omega(z)|)$.
Furthermore, the function 
\beas \frac{1}{(1-z) (1+\alpha z)}
=\frac{1}{\alpha+1}\sum_{n=0}^\infty \left(1+(-1)^n \alpha^{n+1}\right) z^n\eeas
has the positive Taylor coefficients about $z = 0$ for $0\leq \alpha\leq 1/2$. In view of  Schwarz-Pick lemma, we have
\beas (1-|z|^2)|P_f(z)|&=&(1-|z|^2)\left|\frac{\omega'(z)(1+\alpha(\omega(z))^2)}{(1+\alpha \omega(z)-\alpha(\omega(z))^2)(1-\omega(z)) (1+\alpha \omega(z))}\right.\\[2mm]
&&\left.+\frac{\omega(z)}{z(1-\omega(z)) (1+\alpha \omega(z))}\right|\\[2mm]
&\leq& \frac{(1-|z|^2)|\omega'(z)|\left(1+\alpha|\omega(z)|^2\right)}{\left(1-\alpha |\omega(z)|-\alpha|\omega(z)|^2\right)(1-|\omega(z)|) (1+\alpha |\omega(z)|)}\\[2mm]
&&+\frac{(1-|z|^2)|\omega(z)|}{|z|(1-|\omega(z)|) (1+\alpha |\omega(z)|)}\\[2mm]
&\leq& \frac{(1+|\omega(z)|)\left(1+\alpha|\omega(z)|^2\right)}{\left(1-\alpha |\omega(z)|-\alpha|\omega(z)|^2\right)(1+\alpha |\omega(z)|)}\\[2mm]
&&+\frac{(1-|z|^2)|\omega(z)|}{|z|(1-|\omega(z)|) (1+\alpha |\omega(z)|)}.\eeas
For $0< t:=|\omega(z)|\leq |z|<1$, we have
\beas (1-|z|^2)|P_f(z)|\leq  \frac{(1+t)\left(1+\alpha t^2\right)}{\left(1-\alpha t-\alpha t^2\right)(1+\alpha t)}+\frac{(1-|z|^2)t}{|z|(1-t) (1+\alpha t)}.\eeas
Therefore, the pre-Schwarzian norm for functions $f$ in the class $\mathcal{S}^*_{cs}$ is
\bea \label{e7}\Vert P_f\Vert=\sup_{z\in\mathbb{D}}\; (1-|z|^2)|P_f(z)|\leq \sup_{0\leq t\leq |z|<1} G_5(|z|,t),\eea
where 
\beas G_5(r, t)=\frac{(1+t)\left(1+\alpha t^2\right)}{\left(1-\alpha t-\alpha t^2\right)(1+\alpha t)}+\frac{(1-r^2)t}{r(1-t) (1+\alpha t)}\quad\text{for}\quad |z|=r.\eeas
The objective is to determine the supremum of $G_5(r, t)$ on $\Omega=\{(r, t): 0< t\leq r<1\}$.
Differentiating partially $G_5(r, t)$ with respect to $r$, we obtain
\beas\frac{\pa}{\pa r}G_5(r, t)=-\left(\frac{1}{r^2}+1\right)\frac{t}{(1-t) (1+\alpha t)}<0,\eeas
which shows that $G_5(r, t)$ is a monotonically decreasing function of $r\in[t,1)$. Hence, we have $G_5(r, t) \leq G_5(t, t)=G_6(t)$,
where
\beas G_6(t)=\frac{(1+t)\left(1+\alpha t^2\right)}{\left(1-\alpha t-\alpha t^2\right)(1+\alpha t)}+\frac{(1+t)}{(1+\alpha t)},\quad t\in(0,1).\eeas
Differentiate $G_6(t)$ with respect to $t$, we obtain
\beas G_6'(t)&=&\frac{-\alpha^3 t^4+\left(4 \alpha^2-2 \alpha^3\right) t^3+\left(-\alpha^3+7 \alpha^2+2 \alpha\right) t^2+\left(4 \alpha^2+2 \alpha\right) t+2-\alpha}{\left(1-\alpha t-\alpha t^2\right)^2(1+\alpha t)^2}\geq 0\eeas
for $t\in(0, 1)$ and $\alpha\in[0,1/2]$. Therefore, $G_6(t)$ is a monotonically increasing function of $t\in(0,1)$. 
From (\ref{e7}), we have
\beas\Vert P_f\Vert \leq \lim_{t\to1^-}G_6(t)=\frac{2}{\left(1-2\alpha\right)}+\frac{2}{(1+\alpha )}=\frac{2 (2-\alpha)}{(1+\alpha) (1-2 \alpha)}.\eeas
This completes the proof.
\end{proof}
In the following result, we establish the sharp estimate of the pre-Schwarzian norm for the functions in the Ma-Minda type convex class $\mathcal{C}_{con}$.
\begin{theo}
For any $f\in\mathcal{C}_{con}$, the pre-Schwarzian norm satisfies the following sharp inequality
\beas\Vert P_f\Vert\leq \dfrac{6}{3+\alpha}.\eeas
\end{theo}
\begin{proof}
Let $f\in\mathcal{C}_{con}$, it follows from the definition of the class $\mathcal{C}_{con}$ that
\beas 1+\frac{zf''(z)}{f'(z)}\prec \frac{3}{3+(\alpha-3)z-\alpha z^2}.\eeas
Thus, there exists a function $\omega(z)\in\mathcal{B}_0$ such that
\beas1+\frac{zf''(z)}{f'(z)}=\frac{3}{3+(\alpha-3)\omega(z)-\alpha \omega(z)^2}.\eeas
Note that $|\omega(z)|\leq |z|<1$ and from \eqref{g1}, it is easy to see that the function $3/(3+(\alpha-3)z-\alpha z^2)$ has the positive Taylor coefficients about $z = 0$ for $-3<\alpha\leq 1$. Thus, we have 
\beas\frac{3}{3+(\alpha-3)|\omega(z)|-\alpha |\omega(z)|^2}-1&=&\frac{3}{\alpha+3} \sum_{n=1}^\infty \left(1+(-1)^{n}\left(\frac{\alpha}{3}\right)^{n+1}\right)|\omega(z)|^n\\
&\leq&\frac{3}{\alpha+3} \sum_{n=1}^\infty \left(1+(-1)^{n}\left(\frac{\alpha}{3}\right)^{n+1}\right)|z|^n\\
&=&\frac{3}{3+(\alpha-3)|z|-\alpha |z|^2}-1.\eeas
Thus, we have
\beas
(1-|z|^2)|P_f|=(1-|z|^2)\left|\frac{f''(z)}{f'(z)}\right|&=&(1-|z|^2)\left|\frac{1}{z}\left(\frac{3}{3+(\alpha-3)\omega(z)-\alpha \omega(z)^2}-1\right)\right|\\[2mm]
&\leq&\frac{(1-|z|^2)}{|z|}\left(\frac{3}{3+(\alpha-3)|\omega(z)|-\alpha |\omega(z)|^2}-1\right)\\[2mm]
&\leq &(1+|z|)\frac{(3-\alpha+\alpha |z|)}{(3+\alpha |z|)}.\eeas
Therefore, the pre-Schwarzian norm for functions $f$ in the class $\mathcal{C}_{con}$ is
\bea\label{f2} \Vert P_f\Vert=\sup_{z\in\mathbb{D}}(1-|z|^2)|P_f(z)|\leq \sup_{0\leq|z|<1} G_7(|z|),\eea
where 
\beas G_7(r)=\frac{(1+r)(3-\alpha+\alpha r)}{ (3+\alpha r)}\quad \mathrm{for}~|z|=r\in[0, 1).\eeas 
Now we consider the following cases.\\
{\bf Case 1.} If $\alpha=0$, then from (\ref{f2}), we have $G_7(r)=1+r$. Hence, we have $\Vert P_f\Vert\leq 2$.\\
 {\bf Case 2.} If $\alpha\not=0$.
Differentiate $G_7(r)$ with respect to $r$, we obtain
\beas G_7'(r)&=&\frac{\alpha^2 r^2+6 \alpha r+\alpha^2-3 \alpha+9}{(3+\alpha r)^2}.\eeas
The roots of the equation $G_7'(r)=0$ are
\beas r_1=\frac{-3 \alpha-\sqrt{3 \alpha^3- \alpha^4}}{\alpha^2}\quad\text{and}\quad r_2=\frac{-3 \alpha+\sqrt{3 \alpha^3- \alpha^4}}{\alpha^2}.\eeas
It is evident that $3 \alpha^3-\alpha^4\geq 0$ for $0\leq \alpha\leq 1$.\\
{\bf Sub-case 2.1.} If $-3<\alpha <0$, then the equation $G_7'(r)=0$ has no real root and $G_7'(1/3)=(10 \alpha^2-9 \alpha+81)/(\alpha+9)^2>0$ for $-3<\alpha <0$. Hence, $G_7'(r)>0$ for $r\in[0,1)$, which shows that $G_7(r)$ is a monotonically increasing function of $r$. From (\ref{f2}), we have 
\beas \Vert P_f\Vert\leq \lim_{r\to 1^-}G_7(r)=\frac{6}{3+\alpha}.\eeas
{\bf Sub-case 2.2.} Let $0< \alpha\leq 1$.  It is evident that $r_1<0$. We also claim that $r_2<0$. Note that 
$r_2<0$ is equivalent to the inequality
\beas -9 \alpha^2+3 \alpha^3- \alpha^4<0,~\text{which is true for $\alpha\in(0, 1]$.}\eeas
 Therefore, the equation $G_7'(r)=0$ has no positive root in $(0, 1)$ and $G_7'(1/3)=(10 \alpha^2-9 \alpha+81)/(\alpha+9)^2>0$ for $0<\alpha \leq 1$. Hence, $G_7'(r)>0$ for $r\in(0,1)$, which shows that $G_7(r)$ is a monotonically increasing function of $r$. From (\ref{f2}), we have 
\beas \Vert P_f\Vert\leq \lim_{r\to 1^-}G_7(r)=\frac{6}{3+\alpha}.\eeas
\indent To show that the estimate is sharp, let us consider the function $f_3$ defined by
\beas f_3(z)=\int_0^z\mathrm{exp}\left(\int_0^u\frac {3 - \alpha + \alpha t} {(1 - t) (3 + \alpha t)} dt\right)du.\eeas
The pre-Schwarzian norm of $f_3$ is given by
\beas \Vert P_{f_3}\Vert =\sup_{z\in\mathbb{D}}(1-|z|^2)|P_{f_4}(z)|=\sup_{z\in\mathbb{D}}(1-|z|^2)\left|\frac{3-\alpha+\alpha z}{(1-z) (3+\alpha z)} \right|.\eeas
On the positive real axis, we have
\beas \sup_{0\leq r<1}(1-r^2)\frac{3-\alpha+\alpha r}{(1-r) (3+\alpha r)}=
\dfrac{6}{3+\alpha}.\eeas
This completes the proof.
\end{proof}
In the following result, we establish the sharp estimate of the pre-Schwarzian norm for functions in the Ma-Minda type convex class $\mathcal{C}_{lim}$.
\begin{theo}
For any $f\in\mathcal{C}_{lim}$, the pre-Schwarzian norm satisfies the following sharp inequality
\beas \Vert P_f\Vert\leq\left\{\begin{array}{cll}
(1-r_s^2)(1-s+s r_s)&\text{for}\quad s\in(0,1/3],\\[2mm]
1-s&\text{for}\quad s\in[-1/3, 0],
\end{array}\right.\eeas
where $r_s$ is the unique positive root of the equation $3sr^2+2(1-s)r-s=0$ in $(0, 1)$.
\end{theo}
\begin{proof}
Let $f\in\mathcal{C}_{lim}$, it follows from the definition of the class $\mathcal{C}_{lim}$ that
\beas 1+\frac{zf''(z)}{f'(z)}\prec (1+z) (1-s z).\eeas
Thus, there exists a function $\omega(z)\in\mathcal{B}_0$ such that
\beas1+\frac{zf''(z)}{f'(z)}=(1+\omega(z)) (1-s \omega(z)),\quad\text{\it i.e.,}\quad\frac{f''(z)}{f'(z)}=\frac{1}{z}\left((1-s\omega(z))(1+\omega(z))-1\right).\eeas
Note that $|\omega(z)|\leq |z|<1$.
Thus, we have
\beas
(1-|z|^2)|P_f|=(1-|z|^2)\left|\frac{f''(z)}{f'(z)}\right|&=&(1-|z|^2)\left|\frac{1}{z}\left((1-s\omega(z))(1+\omega(z))-1\right)\right|\\
&\leq&\frac{(1-|z|^2)}{|z|}\left((1-s)|\omega(z)|+s|\omega(z)|^2\right)\\
&\leq&(1-|z|^2)(1-s+s |z|).\eeas
Therefore, the pre-Schwarzian norm for the functions $f$ in the class $\mathcal{C}_{lim}$ is
\beas \Vert P_f\Vert=\sup_{z\in\mathbb{D}}(1-|z|^2)|P_f(z)|\leq \sup_{0\leq|z|<1} G_8(|z|),\eeas
where $G_8(r)=(1-r^2)(1-s+s r)$ for $|z|=r\in[0, 1)$.
The following cases occur.\\
{\bf Case 1.} Let $s=0$. Then, $G_8(r)=(1-r^2)$ and hence, we have $\Vert P_f\Vert\leq 1$.\\
{\bf Case 2.} Let $s\in(0,1/3]$. Differentiate $G_8(r)$ with respect to $r$, we obtain
\beas G_8'(r)=-3sr^2-2(1-s)r+s\quad\text{and}\quad G_8''(r)=-6sr-2(1-s)<0.\eeas 
Therefore, $G_8'(r)$ is a monotonically decreasing function of $r$ with
$G_8'(0)=s>0$ and $\lim_{r\to1^-}G_8'(r)=-2<0$. This leads us to the conclusion that the equation $G_8'(r)=0$ has the unique root $r_s$ in $(0,1)$ and thus, 
$G_8(r)$ attains its maximum value at $r_s$. Hence, we have 
$$\Vert P_f\Vert\leq (1-r_s^2)(1-s+s r_s),$$ 
where $r_s$ is the unique positive root of the equation $3sr^2+2(1-s)r-s=0$ in $(0, 1)$.\\
{\bf Case 3.} Let $s\in[-1/3,0)$. Using similar argument as in the \textrm{Case 2}, we have $G_8''(r)<0$ and $G_8'(r)$ is a monotonically decreasing function of $r$.
The roots of the equation $G_8'(r)=0$ are 
\beas r_1(s)=\frac{s-1-\sqrt{4 s^2-2 s+1}}{3 s}\quad\text{and}\quad r_2(s)=\frac{s-1+\sqrt{4 s^2-2 s+1}}{3 s}.\eeas
As $s\in[-1/3, 0)$, it is evident that $4 s^2-2 s+1>1$. We claim that $r_1(s)>1$ and $r_2(s)<0$ for $s\in[-1/3, 0)$. The inequality $r_1(s)>1$ is equivalent to the inequality 
$1+2s+\sqrt{4 s^2-2 s+1}>0$, which is true for $s\in[-1/3, 0)$. Again the inequality $r_2(s)<0$ is equivalent to the inequality $3s^2>0$, which is true for $s\in[-1/3, 0)$. Thus, 
the equation $G_8'(r)=0$ has no roots in $(0, 1)$. Note that $G_8'(1/2)=-(4-5 s)/4<0$ and it follows that $G_8'(r)<0$ for $r\in[0, 1)$.  
Hence, we have $\Vert P_f\Vert\leq 1-s$.\\
From \textrm{Case 1}, \textrm{Case 2} and \textrm{Case 3}, we obtain
\beas \Vert P_f\Vert\leq\left\{\begin{array}{cll}
(1-r_s^2)(1-s+s r_s)&\text{for}\quad s\in(0,1/3],\\[2mm]
1-s&\text{for}\quad s\in[-1/3, 0],
\end{array}\right.\eeas
where $r_s$ is the unique positive root of the equation $3sr^2+2(1-s)r-s=0$ in $(0, 1)$.\\[2mm]
\indent To show that the estimate is sharp, we consider the function $f_4$ given by
\beas f_4(z)=\int_0^z \mathrm{exp}\left(\int_0^u \frac{(1-t) (1+s t)-1}{t}dt\right)du.\eeas
Thus, we have
\beas \frac{f_4''(z)}{f_4'(z)}=\frac{(1-z) (1+s z)-1}{z}=-(1-s+sz)\eeas
The pre-Schwarzian norm of $f_4$ is given by
\beas \Vert P_{f_4}\Vert =\sup_{z\in\mathbb{D}}(1-|z|^2)|P_{f_4}(z)|=\sup_{z\in\mathbb{D}}(1-|z|^2)\left|(1-s+sz)\right|.\eeas
On the positive real axis, we note that
\beas\sup_{0\leq r<1}\left((1-r^2)(1-s+s r)\right)=\left\{\begin{array}{cll}
(1-r_s^2)(1-s+s r_s)&\text{for}\quad s\in(0,1/3],\\[2mm]
1-s&\text{for}\quad s\in[-1/3, 0],
\end{array}\right.\eeas
where $r_s$ is the unique positive root of the equation $3sr^2+2(1-s)r-s=0$ in $(0, 1)$.
This completes the proof.\end{proof}
\noindent In Table \ref{tab2}, we obtain the values of $r_s$ and $\Vert P_f\Vert$ for certain values of $s\in(0, 1/3]$.
\begin{table}[H]
\centering
\begin{tabular}{*{6}{|c}|}
\hline
$s$				&1/3    				&1/5			&     1/7			&2/9	\\
\hline
$r_s$				&0.21525		&0.119633		&0.081666	&0.135042\\
\hline
$\Vert P_f\Vert$&0.704204			&0.812135	&0.863015		&0.793056\\	
\hline
\end{tabular}
\caption{$r_s$ is the unique positive root of $3sr^2+2(1-s)r-s=0$ in $(0, 1)$}
\label{tab2}\end{table}
In the following result, we establish the sharp estimate of the pre-Schwarzian norm for functions in the Ma-Minda type convex class $\mathcal{C}_{cs}$.
\begin{theo}
For any $f\in\mathcal{C}_{cs}$, the pre-Schwarzian norm satisfies the following sharp inequality
\beas\Vert P_f\Vert  \leq \frac{2}{1+\alpha}.\eeas
\end{theo}
\begin{proof}
Let $f\in\mathcal{C}_{cs}$, it follows from the definition of the class $\mathcal{C}_{cs}$ that
\beas 1+\frac{zf''(z)}{f'(z)}\prec 1+\frac{z}{(1-z) (1+\alpha z)}.\eeas
Thus, there exists a function $\omega(z)\in\mathcal{B}_0$ such that
\beas1+\frac{zf''(z)}{f'(z)}=1+\frac{\omega(z)}{(1-\omega(z)) (1+\alpha \omega(z))},\quad\text{\it i.e.,}\quad\frac{f''(z)}{f'(z)}=\frac{\omega(z)}{z(1-\omega(z)) (1+\alpha \omega(z))}.\eeas
Note that $|\omega(z)|\leq |z|<1$ and the function 
\beas \frac{z}{(1-z)(1+\alpha z)}=\frac{1}{1+\alpha} \sum_{n=1}^\infty \left(1-(-1)^{n}\alpha^{n}\right)z^n\eeas
has the positive Taylor coefficients about $z = 0$ for $0\leq \alpha\leq 1/2$. Therefore
\beas\frac{|\omega(z)|}{(1-|\omega(z)|) (1+\alpha |\omega(z)|)}&=&\frac{1}{1+\alpha} \sum_{n=1}^\infty \left(1-(-1)^{n}\alpha^{n}\right)|\omega(z)|^n\\
&\leq&\frac{1}{1+\alpha} \sum_{n=1}^\infty \left(1-(-1)^{n}\alpha^{n}\right)|z|^n=\frac{|z|}{(1-|z|)(1+\alpha |z|)}.\eeas
Thus, we have
\beas
(1-|z|^2)|P_f|=(1-|z|^2)\left|\frac{f''(z)}{f'(z)}\right|&=&(1-|z|^2)\left|\frac{\omega(z)}{z(1-\omega(z)) (1+\alpha \omega(z))}\right|\\
&\leq&(1-|z|^2)\left(\frac{|\omega(z)|}{|z|(1-|\omega(z)|) (1+\alpha |\omega(z)|)}\right)\\
&\leq &\frac{(1-|z|^2)}{(1-|z|) (1+\alpha|z|)}.\eeas
Therefore, the pre-Schwarzian norm for functions $f$ in the class $\mathcal{C}_{cs}$ is
\beas \Vert P_f\Vert=\sup_{z\in\mathbb{D}}(1-|z|^2)|P_f(z)|\leq \sup_{0\leq|z|<1} G_9(|z|),\eeas
where $G_9(r)=(1+r)/(1+\alpha r)$ for $|z|=r\in[0, 1)$.
It is evident that $G_9'(r)=(1-\alpha)/(1+\alpha r)^2>0$ and hence, we have $\Vert P_f\Vert\leq 2/(1+\alpha)$.\\[2mm]
\indent To show that the estimate is sharp, we consider the function $f_5$ given by
\beas f_5(z)=\int_0^z \mathrm{exp}\left(\int_0^u \frac {dt} {(1-t) (1+\alpha t)}\right)du.\eeas
The pre-Schwarzian norm of $f_5$ is given by
\beas \Vert P_{f_5}\Vert =\sup_{z\in\mathbb{D}}(1-|z|^2)|P_{f_5}(z)|=\sup_{z\in\mathbb{D}}(1-|z|^2)\left|\frac{1}{(1-z)(1+\alpha z)}\right|.\eeas
On the positive real axis, we note that
\beas\sup_{0\leq r<1}(1-r^2)\frac{1}{(1-r)(1+\alpha r)}=\frac{2}{1+\alpha}.\eeas
This completes the proof.\end{proof}
\section*{Declarations}
\noindent{\bf Acknowledgment:} The work of the author is supported by University Grants Commission (IN) fellowship (No. F. 44 - 1/2018 (SA - III)).\\[2mm]
{\bf Conflict of Interest:} The author declare that there are no conflicts of interest regarding the publication of this paper.\\[2mm]
{\bf Availability of data and materials:} Not applicable

\end{document}